\documentclass[10pt]{amsart}
\usepackage{amssymb}
\usepackage{color}
\usepackage{graphicx}
\usepackage{float}
\usepackage[all,cmtip]{xy}
\newcommand{\D}{\mathbb{D}}

\newcommand{\SL}{{\mathcal{L}}}

\newcommand{\T}{\mathbb{T}}

\newcommand{\Z}{\mathbb{Z}}

\newcommand{\R}{\mathbb{R}}

\renewcommand{\S}{\mathbb{S}}

\newcommand{\Op}{\operatorname{Op}}

\newcommand{\Cont}{\operatorname{Cont}}
\newcommand{\Diff}{\operatorname{Diff}}

\newtheorem{proposition}{Proposition}
\newtheorem{theorem}[proposition]{Theorem}
\newtheorem{definition}[proposition]{Definition}
\newtheorem{lemma}[proposition]{Lemma}

\newtheorem{corollary}[proposition]{Corollary}

\begin{document}

\title[On the strong orderability of overtwisted 3--folds]{On the strong orderability of overtwisted 3--folds}

\date{June, 2014}


\author{Roger Casals}
\address{Instituto de Ciencias Matem\'aticas -- CSIC--UAM--UC3M--UCM.
C. Nicol\'as Cabrera, 13--15, 28049, Madrid, Spain.}
\email{casals.roger@icmat.es}

\author{Francisco Presas}
\address{Instituto de Ciencias Matem\'aticas -- CSIC--UAM--UC3M--UCM.
C. Nicol\'as Cabrera, 13--15, 28049, Madrid, Spain.}
\email{fpresas@icmat.es}



\begin{abstract}
In this article we address the existence of positive loops of contactomorphisms in overtwisted contact 3--folds. We present a construction of such positive loops in the contact fibered connected sum of certain contact $3$--folds along transverse knots. In particular, we obtain positive loops of contactomorphisms in a class of overtwisted contact structures.
\end{abstract}

\maketitle


Let $(M,\xi)$ be a connected contact manifold with a cooriented contact structure. In \cite{EP}, Y. Eliashberg and L. Polterovich observed that the universal cover $\widetilde{\Cont_0}(M,\xi)$ of the identity component of the group of contactomorphisms carries a natural non--negative normal cone. This structure induces a partial binary relation on the groups $\widetilde{\Cont_0}(M,\xi)$ and $\Cont_0(M,\xi)$. This relation is naturally reflexive and transitive but not necessarily anti--symmetric. In case it is anti--symmetric it provides a partial order on the group of contactomorphims. This has been of central interest \cite{EP,EKP,Gi2} in contact topology.\\

The existence of this partial order in $\widetilde{\Cont_0}(M,\xi)$ can be stated in terms of the non--existence of positive contractible loops of contactomorphisms, confer Section \ref{sec:pre} below. In particular, this leads to the study of positive loops of contactomorphisms and that of positive Legendrian isotopies (see for instance \cite{CFP,CN1,CN2}). A significant part of the current knowledge on the subject can be subsumed as follows. The contact jet spaces $J^1(\R^n)$ and $J^1(\R^n,\S^1)$ along with the spaces of cooriented contact elements do not admit a positive contractible loop of contactomorphisms \cite{Bh,CFP,CN2,EKP,Sa}. The standard contact structure on a sphere $\S^{2n+1}$, different from $\S^1$, does admit a positive contractible loop of contactomorphisms \cite{EKP,Gi2,Ol}.\\

The method used in \cite{CN2} also implies that the space of contact elements of $\T^2$ is strongly orderable, that is, it does not even admit a positive loop of contactomorphisms. In general, \cite[Corollary 9.1]{CN2} implies that the cosphere bundle of a manifold with infinite fundamental group does not admit a positive loop of contactomorphisms. In this direction, P. Weigel \cite{We13} shows that the existence of a non--standard symplectic ball whose Rabinowitz Floer homology growth rate is superlinear can be used to locally perturb any higher--dimensional Liouville fillable contact structure to a strongly orderable contact structure.\\

The canonical contact structures on $J^1(\R)$ and $J^1(\R,\S^1)$ and those obtained as the space of cooriented contact elements of a surface are tight contact structures. Thus the list above does not include any overtwisted contact 3--fold. This article presents the first examples of positive loops in overtwisted contact 3--folds. The first result towards the understanding of positive loops of contactomorphisms in overtwisted 3--folds appears in \cite{CPS}. There, the non--existence of positive loops generated by a Hamiltonian with a small $C^0$--norm has been proven. This statement sided with the folklore conjecture that overtwisted contact manifolds do not admit positive loops of contactomorphisms.

In the present article, we prove that there exist overtwisted contact structures admitting positive loops of contactomorphisms. This is achieved with an explicit construction involving the fibered connected sum with $(\S^1\times\S^2,\xi_{st})$ along a transverse knot. The main result we shall provide is the following

\begin{theorem}\label{thm:main}
Let $(M,\xi)$ be a contact 3--fold that admits a positive loop of contactomorphims $\{\phi_t\}$. Suppose that there exists a locally autonomous orbit $\kappa$ of $\{\phi_t\}$. Then the overtwisted contact 3--fold $(M,\xi^\kappa)$ admits a positive loop of contactomorphisms.
\end{theorem}

In conjuction with \cite[Theorem 1]{CPS}, the Hamiltonians generating such positive loops cannot be $C^0$--small. The notion of a locally autonomous orbit is introduced in Section \ref{sec:pos}. For instance, the Boothby--Wang manifold associated to a surface conforms the hypothesis of Theorem \ref{thm:main}.

\begin{corollary}\label{cor:main1}
Let $(\Sigma,\omega)$ be a symplectic 2--dimensional orbifold. The contact structure obtained by a half Lutz twist along a positive transverse regular fibre of the circle orbibundle $\S(\Sigma,\omega)$ admits a positive loop of contactomorphisms.
\end{corollary}

This yields a positive loop of contactomorphisms for the overtwisted contact structures $(\S^3,\xi_k)$ corresponding to positive integers $k\in\Z^+$ representing the homotopy classes $k\in H^3(M,\pi_3(\S^2))\cong\Z$.


Theorem \ref{thm:main} also applies to the (unique) tight contact structure $\xi_{st}$ on $\S^1\times\S^2$.

\begin{corollary}\label{cor:main3} The contact structure $(\S^1\times\S^2,\xi^\kappa)$ obtained by a half Lutz twist along the positive transverse knot $\kappa=\S^1\times\{(0,0,1)\}\subset(\S^1\times\S^2,\xi_{st})$ admits a positive loop of contactomorphisms.
\end{corollary}

The existence of such positive loops implies squeezing phenomena on the aforementioned contact 3--folds. Nevertheless we cannot conclude its contractibility and thus the squeezing in the isotopy sense does not follow. Similarly, Theorem \ref{thm:main} implies that the binary relation \cite{EP} is not a partial order in $\Cont_0(M,\xi)$ for these overtwisted manifolds, but the lift to the universal cover might still be a partial order. See Subsection \ref{ssec:ord} for details.\\

The article is organized as follows. Section \ref{sec:pre} contains the required preliminaries in contact topology. The construction used in order to prove Theorem \ref{thm:main} involves a fibered connected sum with $\S^1\times\S^2$. Section \ref{sec:pos} presents this contact manifold and describes a certain non--negative loop of contactomorphisms. In Section \ref{sec:sum} we prove Theorem \ref{thm:main} using the tools in Section \ref{sec:pre} and the loop in Section \ref{sec:pos}.\\

{\bf Acknowledgments.} We are grateful to Y. Eliashberg, V. Ginzburg, E. Murphy, L. Polterovich, A. Rechtman and S. Sandon for useful discussions. The present work is part of the authors activities within CAST, a Research Network Program of the European Science Foundation. The authors are supported by the Spanish National Research Project MTM2010--17389.

\section{Preliminaries}\label{sec:pre}

In this section we briefly introduce the basic ingredients involved in Theorem \ref{thm:main}. The material can be essentially extracted from \cite{Ge08} for Subsections \ref{ssec:trans}, \ref{ssec:lutz} and \ref{ssec:loop}. The reader is referred to \cite{EKP,Gi2} for Subsection \ref{ssec:ord}. In this article $(M,\xi)$ denotes a contact 3--fold.

\subsection{Fibered connected sum} \label{ssec:trans}
Let us consider the 3--fold
$$\S^1 \times D^2(R)=\{(\theta;x,y):x^2+y^2\leq R\}=\{(\theta;r,\varphi):r\leq R\}$$
with the contact structure $\xi_0$ defined by the contact form $\alpha_0=d\theta + r^2 d\varphi$.\\

Suppose that $\gamma:\S^1\longrightarrow(M, \xi)$ is a transverse knot with a fixed frame $\tau: \S^1\longrightarrow \gamma^* \xi$. Then, for $R>0$ small enough, there exists a unique (up to contact isotopy) contact embedding
$$\phi: (\S^1 \times D^2(R), \xi_0) \longrightarrow (M, \xi)$$  such that $\phi(\theta,0, 0)= \gamma(\theta)$ and the frame $\phi^*\tau:\S^1\longrightarrow\xi_0$ is homotopic to $\partial_x$.\\

Given two framed transverse knots $(\gamma_1,\tau_1)$ and $(\gamma_2,\tau_2)$ in two contact 3--folds $(M_1,\xi_1)$ and $(M_2,\xi_2)$, we can define the fibered connected sum along these knots. It is described as follows.

Consider the domain $A_{R}=\S^1 \times (-R^2, R^2) \times \S^1$ with coordinates $(\theta, v,\varphi)$ and the contact form $\eta= d\theta +v d\varphi$. Then the pair of the gluing maps:
\begin{eqnarray*}
g_1: \S^1 \times (D^2(R) \setminus \{0\}) & \longrightarrow  & \S^1 \times (0, R^2) \times \S^1 \subset A_{R}\\
(\theta, r, \varphi) & \longmapsto & (\theta, r^2, \varphi)
\end{eqnarray*}
\begin{eqnarray*}
g_2: \S^1 \times (D^2(R) \setminus \{0\}) & \longrightarrow  & \S^1 \times (-R^2,0) \times \S^1 \subset A_{R}\\
(\theta, r, \varphi) & \longmapsto & (\theta, -r^2, -\varphi)
\end{eqnarray*}
satisfy $g_1^*\eta=g_2^*\eta = \alpha_0$ and thus are strict contact embeddings. Then the contact fibered connected sum along $(\gamma_1,\tau_1)$ and $(\gamma_2,\tau_2)$ is the smooth manifold
$$(M_1,\xi_1) \# (M_2,\xi_2):=(M_1\setminus\gamma_1(\S^1))\cup_{g_1\circ \phi_1} A_{R} \cup_{g_2 \circ \phi_2}  (M_2\setminus\gamma_2(\S^1))$$
where $\phi_1$ and $\phi_2$ are the contact embeddings corresponding to $(\gamma_1,\tau_1)$ and $(\gamma_2,\tau_2)$. This 3--fold is endowed with a contact structure in each piece and, since these are glued with $g_1$ and $g_2$, there exists a contact structure on $(M_1,\xi_1) \# (M_2,\xi_2)$.

Observe that an isotopy of framed transverse knots preserves the isotopy class of the resulting contact structure (by Gray's stability). Also, the isotopy class of the contact structure does not depend on each of the frames $(\tau_1,\tau_2)$ but only on their sum:

\begin{lemma}\label{lem:isot}
The contact structure on the fibered connected sum $(M_1,\xi_1)\#(M_2,\xi_2)$ along $(\gamma_1,\tau_1)$ and $(\gamma_2,\tau_2)$ is isotopic to the contact structure on the fibered connected sum $(M_1,\xi_1)\#(M_2,\xi_2)$ along $(\gamma_1,\tau_1+k)$ and $(\gamma_2,\tau_2-k)$ for some $k\in\Z$.
\end{lemma}

The fibered connected sum along a framed transverse knot can be used to modify the contact structure of a 3--fold (while preserving its diffeomorphism type). Indeed, the connected sum $M\#(\S^1\times\S^2)$ along the knot $\S^1\times\{pt.\}$ is diffeomorphic to $M$. This operation yields a non--trivial operation from the contact topology viewpoint, a half Lutz twist.

\subsection{The half Lutz twist}\label{ssec:lutz}
Consider $\S^1\times\R^3$ with coordinates $(\theta;x,y,z) \simeq (\theta;r,\varphi,z)$ and the contact manifold
$$(\S^1\times\S^2,\xi_{st})=(\{(\theta;r,\varphi,z):r^2+z^2=1\},\ker\{zd\theta+r^2d\varphi\})\subset\S^1\times\R^3.$$
This is the unique tight contact structure on $\S^1\times\S^2$, see \cite{Gi1}. 

Let $(\Gamma,\iota)$ be the framed transverse knot on $\S^1 \times \S^2$ defined by $\Gamma(\theta)=(\theta;0,0,1)$ and $\iota(\theta)= \partial_x$.

\begin{definition}
Let $(M, \xi)$ be a contact $3$--fold and $(\gamma,\tau)$ a framed transverse knot. The half Lutz twist of $(M,\xi)$ along the transverse knot $\gamma$ is the contact fibered connected sum $(M,\xi)\# (\S^1\times\S^2,\xi_{st})$ along $(\gamma,\tau)$ and $(\Gamma,\iota)$.
\end{definition}

The half Lutz twist of $(M,\xi)$ along a transverse knot $\gamma$ is denoted by $(M,\xi^\gamma)$. Note that the action of $\Omega SO(3)\subset\Diff(\S^1\times\S^2)$ implies that the diffeomorphism type of $(M,\xi)\# (\S^1\times\S^2,\xi_{st})$ is independent of the choice of frame $\iota$ and thus equal to $M$. In terms of surgeries, it is a Dehn surgery in which the meridian is sent to the meridian and thus the smooth type of the resulting manifold remains the same. Similarly, the contact structure $\xi^\gamma$ does not depend either on the choice of frame, see Lemma \ref{lem:isot} or \cite{E1,Ge}.\\

There are two relevant features regarding $(M,\xi^\gamma)$. First, it is an overtwisted contact 3--fold. There is a family of overtwisted disks that appear from the family of immersed overtwisted disks $\{Ê\theta \} \times \S^2$ in the tight $(\S^1\times\S^2,\xi_{st})$, whose boundaries (collapsed at a point) form the knot $\Gamma$. Second, the homotopy class of $\xi^\gamma$ differs from that of $\xi$. The primary obstruction is the class $d^2(\xi,\xi^\gamma)=c_1(\xi)-c_1(\xi^\gamma)=-2PD([\gamma])\in H^2(M,\pi_2(\S^2))$.\\

The positive loop of contactomorphims obtained in Theorem \ref{thm:main} is essentially built separately in the two pieces of a fibered connected sum. One corresponds to the given contact 3--fold $(M,\xi)$ and the other belongs to $(\S^1\times\S^2,\xi_{st})$. The loop is constructed by gluing a positive loop in each of the pieces, thus resulting in a positive loop of contactomorphims for the half Lutz twist of $(M,\xi)$.

\subsection{Loops of contactomorphisms}\label{ssec:loop}

Let $(M, \xi)$ be a contact structure and a $1$--form $\alpha$ such that $\xi= \ker \alpha$. The choice of $\alpha$ uniquely determines a vector field $R_\alpha$ such that
$$i_{R_\alpha}\alpha=1,\quad i_{R_\alpha}d\alpha=0.$$
A vector field $X$ is said to be contact if $\SL_X \alpha =f\alpha$, for some $f\in C^\infty(M)$. Given a contact vector field $X$, the function $H=\alpha(X)\in C^\infty(M)$ satisfies the equations
\begin{eqnarray*}
i_X \alpha & = & H, \\
i_X d\alpha &= &  (d_{R_\alpha}H) \alpha - dH.
\end{eqnarray*}
Conversely, given a function $H\in C^\infty(M)$ there exists a unique contact vector field $X$ verifying the equations above. The function $H$ is called the Hamiltonian function associated to $X$. This establishes a linear isomorphism (depending on $\alpha$) between the vector space of contact vector fields and the vector space of smooth functions.\\

The correspondance can be made time--dependent. Given a time--dependent flow $\phi_t:M\times[0,1]\longrightarrow M$ of contactomorphisms, its associated time--dependent vector field is defined by
$$\dot\phi_t = X_t \circ \phi_t.$$
The function $H_t=\alpha(X_t): M \times [0,1]\longrightarrow\R$ will be referred to as the Hamiltonian generating the contact flow $\phi_t$, and denoted by $H(\phi_t)$. It will we assumed to be 1--periodic in time. The flow of contactomorphisms $\phi_t$ is said to be a smooth loop if $\phi_1=id$ and the quotient map $\phi_t:M \times \S^1 \longrightarrow M$ is smooth. The loop of contactomorphisms is positive if its generating Hamiltonian is positive, i.e. $H_t(p,t)>0$ at any $(p,t)\in M\times\S^1$. \\

There are two useful operations in the spaces of loops of contactomorphisms: concatenation and composition. The concatenation is defined as follows. Let $\{\Phi^1_t,\ldots,\Phi^l_s\}$ be a set of $l\in\Z^+$ loops of contactomorphisms respectively generated by Hamiltonians $\{F^1_t,\ldots,F^l_t\}$.\\

The concatenation of the loops $\{\Phi^1_t,\ldots,\Phi^l_t\}$ is defined as
$$
\Phi^1_t\odot\ldots\odot\Phi^l_t = \left\{ \begin{array}{llllll}
\Phi^1_{lt} & t\in[0,1/l], \\
\Phi^2_{lt-1} & t\in[1/l, 2/l], \\
 & \vdots \\
\Phi^{l-1}_{lt-l+2} & t\in[1-2/l,1-1/l],\\
\Phi^l_{lt-l+1} & t\in[1-1/l,1].
\end{array}\right.
$$
The generating Hamiltonian $C:M\times\S^1\longrightarrow\R$ for the concatenation is
\begin{equation*}
C_t=H(\Phi^1_t\odot\ldots\odot\Phi^l_t) = \left\{
\begin{array}{llllll}
F^1(\cdot,{lt}) & t\in[0,1/l], \\
F^2(\cdot,lt-1) & t\in[1/l, 2/l], \\
 & \vdots \\
F^{l-1}(\cdot,lt-l+2) & t\in[1-2/l,1-1/l],\\
F^l(\cdot,lt-l+1) & t\in[1-1/l,1].
\end{array}\right.
\label{eq:conca}
\end{equation*}

Let $\Phi_t$ and $\Psi_t$ be two loops of contactomorphisms generated by $F_t$ and $G_t$. The second operation is the composition $\{\Phi_t\circ\Psi_t\}_t$ of $\Phi_t$ and $\Psi_t$. Suppose that the first loop satisfies $\Phi_t^* \alpha = e^{f_t} \alpha$, then the Hamiltonian generating the composition is
\begin{equation*}
H(\Phi_t \circ \Psi_t)(p,t)=F_t(p,t) + e^{-f_t} G_t(\Phi_t^{-1}(p), t). \label{eq:compo}
\end{equation*}
In addition, the conjugation $\{\psi \circ \Phi_t \circ  \psi^{-1}\}_t$ of the loop $\Phi_t$ by a contactomorphism $\psi\in\Cont(M,\xi)$, such that $\psi^* \alpha = e^f \alpha$, is a loop of contactomorphisms generated by the Hamiltonian
\begin{equation*}
H(\psi \circ \Phi_t \circ  \psi^{-1})(p,t) =e^{-f}F_t(\psi^{-1}(p),t). \label{eq:conj}
\end{equation*}

\subsection{Orderability}\label{ssec:ord}
Let us consider the identity component of the group of contactomorphisms $G=\Cont_0(M,\xi)$ and its universal cover $\widetilde{G}=\widetilde{\Cont}_0(M,\xi)$. These groups are endowed with a natural relation. Given $f,g\in\Cont_0(M,\xi)$, the relation is defined as $f\geq g$ if and only if there exists a path $\phi_t$ of contactomorphisms such that $\phi_1=f\circ g^{-1}$ and its generating Hamiltonian is non--negative. This relation is reflexive and transitive. Similarly, given two elements $[\phi_t],[\psi_t]\in\widetilde{G}$. The relation $[\phi_t]\geq[\psi_t]$ if and only if $[\phi_t\circ\psi^{-1}_t]$ admits a representative generated by a non--negative Hamiltonian is reflexive and transitive.\\

The contact manifold $(M,\xi)$ is said to be strongly orderable if the relation $(G,\geq)$ is antisymmetric (and thus defines a genuine partial order). It is said to be orderable if the relation $(\widetilde{G},\geq)$ is also antisymmetric. The following criterion relates the existence of this genuine partial order with the existence of positive loops of contactomorphisms:

\begin{proposition}\cite[Criterion 1.2.C]{EP} The relation $\geq$ is a non--trivial partial order on $G$ if and only if there are no loops of contactomorphisms of $(M,\xi)$ generated by a strictly positive Hamiltonian.\\

In addition, the relation $\geq$ is a non--trivial partial order on $\widetilde{G}$ if and only if there are no contractible loops of contactomorphisms of $(M,\xi)$ generated by a strictly positive Hamiltonian.
\end{proposition}

Theorem \ref{thm:main} implies the existence of non--strongly orderable overtwisted contact 3--folds. These are the first examples relating overtwisted 3--folds to orderability.

\section{Locally autonomous loops of contactomorphisms}\label{sec:pos}

Let $(M,\xi)$ be a contact 3--fold and $\alpha$ an associated contact form. The fibered connected sum along a tranverse knot has been described in Subsection \ref{ssec:trans}. The aim of this section is to introduce the property of a positive loop of contactomorphisms that allows us to obtain a loop of contactomorphisms in the fibered connected sum $(M,\xi)\#(\S^1\times\S^2,\xi_{st})$.\\

This appropriate class of loops are the locally autonomous loops, described as follows. Let $p\in M$ be a point and $\{\phi_{t}\}$ a positive loop of contactomorphisms generated by a Hamiltonian $F_t$.
\begin{definition}\label{def:loc}
The loop $\{\phi_t\}$ is said to be locally autonomous at $p$ if there exists $\Op(p)$ such that
$$F(\phi_t(q),t_0)=F(\phi_t(q), t_1),\quad \forall t, t_0, t_1 \in \S^1\mbox{ and }\forall q\in \Op(p)$$
and the map $\phi_t(p):\S^1\longrightarrow M$ is an embedding.
\end{definition}
Observe that this definition does not depend on the choice of contact form $\alpha$ for $\xi$. The local autonomy at $p$ is equivalent to $\dot{\phi}_t$ being time--independent on the trajectories passing through $\Op(p)$ and a positive loop that is locally autonomous at any point of the manifold is time--independent.\\

There exists also a normal form for the contact structure in a neighborhood of the orbit of the point $p$. It is used in order to glue the dynamics in a fibered connected sum. The normal form is the content of the following

\begin{proposition}\label{lem:key}
Let $\{\phi_t\}$ be a locally autonomous loop around $p$. Then there exist a constant $\rho\in\R^+$, a tubular neighborhood $T_p$ of the orbit through $p$ and a contactomorphism
$$\psi:(\S^1\times D^2(\rho), \ker\{\alpha_0= d\theta+r^2 d\varphi \})\longrightarrow(T_p,\xi|_{T_p})\mbox{ such that }\alpha_0(\psi^*\dot{\phi_t})=1.$$
\end{proposition}
\begin{proof}
Consider the contact form $\eta=\alpha/F_0$. The contact Hamiltonian $H_t$ associated to the loop $\{\phi_t\}$ with respect to $\eta$ satisfies $H_0=1$, and thus the contact vector field $X_t$ coincides with the Reeb field $R_\eta$ at $t=0$. The strict Darboux Theorem (\cite[Section 2.5]{Ge}) implies the existence of a constant $\rho\in\R^+$, a neighborhood $U_p$ and a strict contactomorphism
$$f:((-\varepsilon,\varepsilon)\times D^{2n}(\rho),\alpha_0)\longrightarrow (U_p,\eta).$$
We can suppose that the neighborhood $U_p$ is contained in the neighborhood $\Op(p)$ provided by Definition \ref{def:loc}. Since the Reeb flow is a strict contactomorphism and the flow $\phi_t$ is locally autonomous on $\Op(p)$, the flow $\phi_t$ coincides with the Reeb flow in $\Op(p)$, and hence it is a strict contact flow.\\

The positive loop $\Psi_t$ generated by the Reeb field on $(\S^1\times D^{2n}(\rho),\ker\{\alpha_0\})$ is $\Psi_t(\theta,x)=(\theta+t,x)$. The neighborhood $T_p$ is obtained through the flow of $U_p$ and the contactomorphism $\psi$ is the strict contact embedding
\begin{eqnarray*}
\psi:\S^1\times D^{2n}(\rho) &\longrightarrow&  M \\
(\theta,x) & \longmapsto & \phi_{\theta}(f(\Psi_{-\theta}(\theta,x))).
\end{eqnarray*}
\end{proof}

Consider two contact 3--folds $(M,\xi)$ and $(N,\eta)$ with positive loops of contactomorphisms $\Phi_t$ and $\Psi_t$ locally autonomous at $p\in M$ and $q\in N$. Let $\gamma$ and $\kappa$ be the orbits of $p$ and $q$ with respect to $\Phi_t$ and $\Psi_t$, which come equipped with natural framings provided by Proposition \ref{lem:key} . The fibered connected sum along $\gamma$ and $\kappa$, introduced in Subsection \ref{ssec:trans}, admits a positive loop of contactomorphisms.\\

This positive loop is defined as $\Phi_t$ and $\Psi_t$ in $M\setminus\Op(\gamma)$ and $N\setminus\Op(\kappa)$, considered as submanifolds of $(M,\xi)\#(N,\eta)$ and extended to the gluing region of $(M,\xi)\#(N,\eta)$ with each of the two loops of contactomorphisms. In detail, Proposition \ref{lem:key} provides a normal form for both neighborhoods $\Op(\gamma)$ and $\Op(\kappa)$. This allows us to glue the two corresponding Hamiltonians $H(\Phi_t)$ and $H(\Psi_t)$ in their local normal form, both being constant on the gluing region and thus coinciding at $\S^1\times\{0\}\times\S^1\subset A_R$.  This positive loop of contactomorphisms of $(M,\xi)\#(N,\eta)$ is denoted by $\Phi_t\#\Psi_t$.\\


The proof of Theorem \ref{thm:main} consists of this construction applied to the manifold $(\S^1\times\S^2,\xi_{st})$ with an appropriate loop of contactomorphisms. The overtwistedness of the resulting contact structure follows from Subsection \ref{ssec:lutz}. Section \ref{sec:sum} provides this loop and concludes Theorem \ref{thm:main}.

\section{Proof of the main result}\label{sec:sum}

In this section Propositions \ref{prop:s1s2} and \ref{prop:autos1s2} are used to prove Theorem \ref{thm:main}.

Consider coordinates $(\theta;r,\varphi,z)\in\S^1\times\R^3$ and the contact form $\alpha_{st}=zd\theta+r^2d\varphi$ on the manifold $\S^1\times\S^2=\{(\theta;r,\varphi,z):r^2+z^2=1\}\subset\S^1\times\R^3$. We can define the two solid tori
$$\T_1=\S^1\times\D^2=\{(\theta;r,\varphi,z):r^2+z^2=1,z\geq0\},\quad \T_2=\S^1\times\D^2=\{(\theta;r,\varphi,z):r^2+z^2=1,z\leq0\}.$$
There exists a non--negative autonomous Hamiltonian $R_t:\S^1\times\S^2\longrightarrow\R$ defined as $R_t(\theta;r,\varphi,z)=r^2$ which generates the non--negative loop of contactomorphisms $\{\rho_t\}$ given by
$$\rho_t(\theta;r,\varphi,z)=(\theta;r,\varphi+t,z).$$

A second autonomous Hamiltonian is also central to our construction. It is the Hamiltonian $Z_t:\S^1\times\S^2\longrightarrow\R$ defined as $Z_t(\theta;r,\varphi,z)=z$ whose associated loop of contactomorphisms $\{\zeta_t\}$ is $$\zeta_t(\theta;r,\varphi,z)=(\theta+t;r,\varphi,z).$$
It is a loop of strict contactomorphisms, i.e. $\zeta_t^* \alpha_{st} = \alpha_{st}$. \\

The loops $\rho_t$ and $\zeta_t$ are autonomous and commute, however only $\rho_t$ is non--negative. Let us construct a locally autonomous positive loop of contactomorphisms in $(\S^1 \times \S^2,\xi_{st})$. It is obtained in two steps corresponding to the two subsequent Propositions.

\begin{proposition}\label{prop:s1s2}
There exists a loop $\{\beta_t\}\in\Omega\Cont(\S^1\times\S^2,\xi_{st})$ which coincides with $\{\rho_t\circ\rho_t \}$ on the solid torus $\T_1$ and it is positive on the solid torus $\T_2$.
\end{proposition}
The proof follows closely the argument of \cite[Prop. 2.1.B]{EP} and \cite[Prop. 2.3]{Gi2}.
\begin{proof}
Consider the transverse knot $\gamma(\theta)=(-\theta;0,0, -1)$ in $(\S^1\times\S^2,\xi_{st})$. Suppose that for a small enough neighborhood $\Op(\gamma)$ there exists a contactomorphism $\psi\in\Cont(\S^1\times\S^2,\xi_{st})$ supported in $\Op(\gamma)$ and such that $\gamma\cap\psi(\gamma)=\emptyset$. Then the loop $\beta_t =\rho_t\circ\psi\circ\rho_t\circ\psi^{-1}$ coincides with $\rho_t\circ\rho_t$ on $\T_1$ and its Hamiltonian
$$H(\beta_t)=R_t(p,t)+H(\psi\circ\rho_t\circ\psi^{-1})(\rho_t^{-1}(p),t)=R_t(p,t)+e^{-f}R_t((\rho_t\circ\psi)^{-1}(p),t)$$
is positive on $\T_2$ since at least one of the two summands is strictly positive. In the above formula $f\in C^\infty(M)$ is such that $\psi^*\alpha=f\alpha$, and confer Subsection \ref{ssec:loop} for the expression of the Hamiltonian. Let us show the existence of the contactomorphism $\psi$.\\

Let $\varepsilon\in\R^+$ be small enough, $(\theta, x,y)\in \S^1 \times D^2_\varepsilon$ local coordinates and $g:\S^1 \times D^2_\varepsilon\longrightarrow\Op(\gamma)$ a local chart such that $\ker g^* \alpha_{st}= \ker\{d\theta+xdy\}$. It suffices to construct the compactly supported contactomorphism $\psi$ in this local model $\S^1\times D^2_\varepsilon$. The contact vector field $\partial_y$ is generated by the Hamiltonian $H(\theta; x,y)= x$. This Hamiltonian can be cut--off to a smooth Hamiltonian
$$\widetilde{H}:\S^1\times D^2_\varepsilon\longrightarrow\R\mbox{ such that } \widetilde{H}= H\mbox{ on }\S^1\times D^2_{\varepsilon/4}\mbox{ and  }\widetilde{H}=0\mbox{ on }\S^1 \times (D^2_{\varepsilon} \setminus  D^2_{3\varepsilon/4}).$$
The flow generated by $\widetilde{H}$ exists for $\tau\in\R^+$ small enough, and for one such $\tau$ we can define the contactomorphism $\psi$ to be the $\tau$--time flow.\end{proof}
%
%

\begin{proposition}\label{prop:autos1s2}
The loop $\delta_t=\zeta_t \circ (\beta_t \odot \stackrel{k)}{\cdots} \odot \beta_t)$ in $\Cont_0(\S^1\times\S^2,\xi_{st})$ is locally autonomous at any point of the open set $\mathring{\T}_1$ and positive on $(\S^1\times\S^2,\xi_{st})$ for $k\in\Z^+$ large enough.
\end{proposition}
\begin{proof}
The loop $\zeta_t$ preserves the decomposition $\S^1\times\S^2=\T_1\cup\T_2$. The Hamiltonian associated to the loop $\{\beta_t \odot \stackrel{k)}{\cdots}\odot\beta_t\}$ is the smooth function $kH(\beta_{kt})$. The Hamiltonian $H(\beta_t)$ is positive on $\T^2$ and thus for $k$ large enough
$$H(\delta_t)(p,t)=z(p)+kH(\beta_{kt})(\zeta_t^{-1}(p),t)\geq -1+kH(\beta_{kt})(\zeta_t^{-1}(p),t)>0.$$
Therefore the Hamiltonian $H(\delta_t)$ is positive in $\T_2$.\\

In the solid torus $\T_1$, the Hamiltonian $H(\delta_t)|_{\T_1}(\theta;r,\varphi,z)=z+2kr^2$ is positive, autonomous and its flow preserves $\T_1$. This concludes the statement.
\end{proof}
The existence of the loop $\delta_t\in\Omega\Cont(\S^1\times\S^2,\xi_{st})$ in Proposition \ref{prop:autos1s2} implies Theorem \ref{thm:main}.

{\it Proof of Theorem \ref{thm:main}}: The loop of contactomorphisms $\delta_t$ constructed in Proposition \ref{prop:autos1s2} is positive and locally autonomous at $p=(0;0,0,1)$. Consider the transverse knot $\gamma=\{z=1\}=\{(\theta;0,0,1)\}$, this is coincides with the orbit of $\delta_t$ at $p$. Then the loop of contactomorphisms $\phi_t\#\delta_t$ of the fibered connected sum $(M,\xi)\#(\S^1\times\S^2,\xi_{st})$ along $\kappa\#\gamma$ is generated by a positive Hamiltonian. Subsection \ref{ssec:lutz} implies that the construction does not depend on the choice of frames and the resulting contact manifold is $(M,\xi^\kappa)$.\hfill$\Box$

The geometric argument used to prove Theorem \ref{thm:main} should apply to higher dimensional contact manifolds. There is however no explicit example of an overtwisted contact manifold of higher dimension and thus there is no local model in order to glue (neither a general notion of a higher dimensional Lutz twist). 


 

\end{document}